\documentclass[english,12pt]{article}
\usepackage{amssymb}

\usepackage{bm}
\usepackage{amsmath}
\usepackage{amsfonts}
\usepackage{mathrsfs}
\usepackage{graphicx}
\usepackage{tikz}
\usetikzlibrary{arrows,shapes,positioning}
\usetikzlibrary{decorations.markings}
\tikzstyle arrowstyle=[scale=1.5]
\tikzstyle directed=[postaction={decorate,decoration={markings, mark=at position .9 with {\arrow[arrowstyle]{stealth}}}}]
\tikzstyle reverse directed=[postaction={decorate,decoration={markings, mark=at position .9 with {\arrowreversed[arrowstyle]{stealth};}}}]
\usepackage{}

\newtheorem{Theorem}{Theorem}[section]
\newtheorem{Corollary}{Corollary}[section]

\newtheorem{Lemma}{Lemma}[section]

\def\emptyset{\mbox{{\rm \O}}}

\newenvironment{proof}{
\noindent {\bf Proof.}\rm}%
{\mbox{}\hfill\rule{0.5em}{0.809em}\par}
\usepackage{calc}
\usepackage{tikz}
\usetikzlibrary{decorations.markings}
\tikzstyle{vertex}=[circle, draw, inner sep=0pt, minimum size=6pt]

\begin{document}

\title{\LARGE{\textbf{Sufficient conditions for closed-trailable in digraphs\thanks{This work was supported by
 the Natural Science Foundation of Xinjiang (No. 2020D04046) and the National Natural Science Foundation of China (No. 12261016, No. 11901498, No. 12261085).}}}}
\author{{ Changchang Dong$^{a}$, Hong Yang$^{b}$, Jixiang Meng$^{b}$, Juan Liu$^{c}$\footnote{Corresponding
author.
E-mail:  liujuan1999@126.com}}
\\
\small a. Department of Mathematical Sciences, Tsinghua University, Beijing 100084, China
\\
\small b. College of Mathematics and System Sciences, Xinjiang University, Urumqi 830046, China  \\
\small  c. College of Big Data Statistics, Guizhou University of Finance and Economics, Guiyang 550025, China  \\}

\date{}
\maketitle

\begin{abstract}
A digraph $D$ with a subset $S$ of $V(D)$ is called $\boldsymbol{S}${\bf -strong} if for every pair of distinct
vertices $u$ and $v$ of $S$, there is a $(u, v)$-dipath and a $(v, u)$-dipath in $D$.
 We define a digraph $D$ with a subset $S$ of $V(D)$ to be $\boldsymbol{S}${\bf -strictly strong}
 if
 there exist
 two nonadjacent
vertices $u,v\in S$ such that $D$ contains a closed ditrail through the vertices $u$ and $v$;
and define a subset $S\subseteq V(D)$ to be {\bf closed-trailable} if $D$ contains a closed ditrail
through all the vertices of $S$.
In this paper, we prove that for a digraph $D$ with $n$ vertices and a subset $S$ of $V(D)$,
 if $D$ is $S$-strong and if
 $d(u) + d(v)\geq 2n -3$ for any two nonadjacent vertices $u,v$ of $S$, then $S$ is closed-trailable.
 This result generalizes the theorem of Bang-Jensen et al. \cite{BaMa14} on supereulerianity.
Moveover, we show that for a digraph $D$ and a subset $S$ of $V(D)$,
 if $D$ is $S$-strictly strong and if
 $\delta^0(D\langle S\rangle)\geq\alpha'(D\langle S\rangle)>0$, where $\delta^0(D\langle S\rangle)$ is
 the minimum semi-degree of $D\langle S\rangle$ and $\alpha'(D\langle S\rangle)$ is the matching number of $D\langle S\rangle$, then $S$ is closed-trailable.
 This result generalizes the theorem of Algefari et al. \cite{AlLa15} on supereulerianity.
\end{abstract}

\vskip 0.1 in

\section{Introduction}

We consider strict digraphs and simple graphs.
We refer the reader to
\cite{BaGu09} for digraphs and \cite{BoMu08} for graphs for undefined terms and notations.
In this paper, we
define $[k]=\{1,2,\ldots ,k\}$ for an integer $k>0$;
 use $(u,v)$ to denote an arc oriented from a vertex
$u$ to a vertex $v$; and use $[u, v]$ to denote an arc which is either $(u, v)$ or $(v, u)$. When $[u, v] \in A(D)$, we say
that $u$ and $v$ are adjacent.
We often write {\bf dipaths} for directed paths, {\bf dicycles} for directed cycles and {\bf ditrails} for directed trails in digraphs.
The {\bf length} of a ditrail is the number of its arcs.
If a ditrail $T$ from $w$ to $z$, we may call it $(w,z)$-ditrail $T$ or denote it by $T_{[w,z]}$, and say
$w$
is the {\bf initial vertex} of $T$ and $z$ is the {\bf terminal vertex} of $T$.
We often 
 use $K_n^*$ to represent the {\bf complete digraph} with $n$ vertices.
 For a graph $G$, a digraph $D$ is called a {\bf biorientation} of $G$ if $D$ is obtained from $G$ by replacing each edge $xy$ of $G$ by either $(x,y)$ or $(y,x)$ or the pair $(x,y)$ and $(y,x)$. Recall that a {\bf semicomplete digraph} is a biorientation of a complete graph.
For a subset $X\subseteq V(D)$, the subdigraph {\bf induced} by $X$ is the digraph
$D\langle X\rangle = (X,A')$, where $A'$ is the set of arcs in $A(D)$ which have both end vertices in $X$.
For a subset $A'\subseteq A(D)$, the subdigraph {\bf arc-induced} by $A'$ is the digraph
$D\langle A'\rangle = (V',A')$, where $V'$ is the set of vertices in $V(D)$ which are incident with at least one arc from $A'$.

Let $T=v_{1} v_{2} \cdots v_{k}$ denote a ditrail.
For any $1 \leq i \leq j \leq k$, we use $T_{[v_{i}, v_{j}]}=v_{i} v_{i+1} \cdots v_{j-1} v_{j}$ to denote the {\bf sub-ditrail} 
of $T$. Likewise, if $Q=u_{1} u_{2} \cdots u_{k} u_{1}$ is a closed ditrail, then for any $i, j$ with $1 \leq i<j \leq k, Q_{[u_{i}, u_{j}]}$ denotes the sub-ditrail $u_{i} u_{i+1} \cdots u_{j-1} u_{j}$. If $T^{\prime}=w_{1} w_{2} \cdots w_{k^{\prime}}$ is a ditrail with $v_{k}=w_{1}$ and $V(T) \cap V\left(T^{\prime}\right)=\left\{v_{k}\right\}$, then we use $T T^{\prime}$ or $T_{[v_{1}, v_{k}]} T^{\prime}_{[v_{k}, w_{k^{\prime}}]}$ to denote the ditrail $v_{1} v_{2} \cdots v_{k} w_{2} \cdots w_{k^{\prime}}$. If $V(T) \cap V\left(T^{\prime}\right)=\emptyset$ and there is a dipath $z_{1} z_{2} \cdots z_{t}$ with
$z_{2}, \cdots, z_{t-1} \notin V(T) \cup V\left(T^{\prime}\right)$ and with $z_{1}=v_{k}$ and $z_{t}=w_{1}$, then we use $Tz_{1} \cdots z_{t} T^{\prime}$ to denote the ditrail $v_{1} v_{2} \cdots v_{k} z_{2} \cdots z_{t} w_{2} \cdots w_{k^{\prime}}$. In particular, if $T$ is a $(v, w)$-ditrail of a digraph $D$ and $(u ,v), (w, z) \in A(D)-A(T)$, then we use $u v T w z$ to denote the $(u, z)$-ditrail $D\langle A(T) \cup\{(u ,v), (w, z)\}\rangle$. The subdigraphs $u v T$ and $T w z$ are similarly defined.

For a digraph $D$, $a\in A(D)$ and a subdigraph $S$ of $D$,
we use $D-S$ to denote the subdigraph $D\langle V(D)-V(S)\rangle$,
use $D-a$ to denote the subdigraph $D\langle A(D)-\{a\}\rangle$,
and use $D+a$ to denote the subdigraph $D\langle A(D)+\{a\}\rangle$.
Let $D_{1}$ and $D_{2}$ be two digraphs, the \emph{union} $D_{1} \cup D_{2}$ of $D_{1}$ and $D_{2}$ is a digraph with vertex set $V\left(D_{1} \cup D_{2}\right)=V\left(D_{1}\right) \cup V\left(D_{2}\right)$ and arc set $A\left(D_{1} \cup D_{2}\right)=A\left(D_{1}\right) \cup A\left(D_{2}\right)$.

A digraph $D$ is {\bf strong} if for every pair of distinct
vertices $u$ and $v$ of $D$, there is a $(u, v)$-dipath and a $(v, u)$-dipath in $D$.
Following \cite{LiFS07}, a digraph $D$ with a subset $S$ of $V(D)$ is called $\boldsymbol{S}${\bf -strong} if for every pair of distinct
vertices $u$ and $v$ of $S$, there is a $(u, v)$-dipath and a $(v, u)$-dipath in $D$.
We define a digraph $D$ with a subset $S$ of $V(D)$
to be $\boldsymbol{S}${\bf -strictly strong}
 if
 there exist
 two nonadjacent
vertices $u,v\in S$ such that $D$ contains a closed ditrail through the vertices $u$ and $v$.
Clearly, if $D$ is strong, then $D$ is $S$-strong; but if $D$ is $S$-strong, then $D$ is not necessarily strong.

Let $\kappa(G), \kappa'(G), \alpha(G)$, and $\alpha'(G)$ denote the connectivity, the
edge connectivity, the independence number, and the matching number of a graph $G$; and
$\kappa(D)$ and $\lambda(D)$ denotes the {\bf vertex-strong connectivity} and the {\bf arc-strong connectivity} of a
digraph $D$, respectively. As it is implied by Corollary 5.4.3 of \cite{BaGu09}, we have $\lambda(D)\geq\kappa(D)$.
Let $UG(D)$ denote the {\bf underlying graph} of $D$.
 The {\bf independence
number} and the {\bf matching number} of a digraph $D$ are defined as
$\alpha(D) = \alpha(UG(D))$ and $\alpha'(D) = \alpha'(UG(D))$,
respectively.

Following \cite{BaGu09}, for $S, T \subseteq V(D)$, define \[(S, T)_D = \{(s,t) \in A(D)\::\; s \in S, t \in T\}.\]
Let $N_D^+(v) =
\{u \in V(D)-v: (v,u) \in A(D)\}$, $N_D^-(v) = \{u \in V(D)-v:
(u,v) \in A(D)\}$ and $N_D(v) =N_D^+(v)\cup N_D^-(v) $ denote the {\bf out-neighbourhood}, the {\bf in-neighbourhood} and the {\bf neighbourhood} of $v$ in $D$.
Let $d_D^-(v)=|N_D^-(v)|,d_D^+(v)=N_D^+(v)$ and $d_D(v) = d_D^-(v)+d_D^+(v)$
 denotes, respectively, the {\bf in-degree}, the {\bf out-degree} and the {\bf degree}
 of a vertex $v\in V(D)$.
 Let $\delta^+(D) = min\{d^+_D (v): v \in V(D)\},\delta ^-(D) = min\{d^-_D (v): v \in V(D)\}$ and $\delta^0(D) = min\{\delta^+(D),\delta^-(D)\}$ denotes, respectively, the {\bf minimum in-degree}, the {\bf minimum out-degree} and the {\bf minimum semi-degree} 
 of a digraph $D$.
  For a subdigraph $H$ of $D$ and a vertex $v\in V(D)$, define $N_H^+(v) =
\{u \in V(H)-v: (v,u) \in A(D)\}$, $N_H^-(v) = \{u \in V(H)-v:
(u,v) \in A(D)\},N_H(v) =N_H^+(v)\cup N_H^-(v), d_H^-(v)=|N_H^-(v)|,d_H^+(v)=|N_H^+(v)|$ and $d_H(v) = d_H^-(v)+d_H^+(v)$.
When $H\subseteq V(D)$, similarly define $N_H^+(v),N_H^-(v),N_H(v),d_H^-(v),d_H^+(v)$ and $d_H(v)$.
 For a subset
$X\subseteq V(D)$, define $N_D(X) =\cup_{x\in X}N_D(x)\setminus X$.
When the digraph $D$ is understood from the context, we often omit the subscript $D$.



Let $M$ be a matching in
a graph $G$, a path $P$ is an
$\boldsymbol{M}${\bf -augmenting path}
 if the edges of $P$ are alternately in $E(G)- M$ and in $M$, and if both
end vertices of $P$ are not in $V(M)$. The following theorem is fundamental.
\begin{Theorem}\label{M-P iff} (\cite{Ber57})
A matching $M$ in $G$ is a maximum matching if and only if $G$
contains no $M$-augmenting paths.
\end{Theorem}

 A digraph $D$
 is {\bf supereulerian} if
it has a spanning eulerian subdigraph, or equivalently, $D$ contains a spanning closed ditrail.
A set $S$ of vertices of a graph (digraph) $G$ is said to be {\bf cyclable} in $G$ if $G$ contains a cycle (dicycle)
through all the vertices of $S$. This definition was first introduced by Ota in \cite{Ota95}.
 In \cite{LiFS07}, Li et al. proved the following for cyclability in digraphs. 
\begin{Theorem}\label{LiFS07} (\cite{LiFS07})
Let $D$ be a digraph with $n$ vertices and $S\subseteq V (D)$. If $D$ is $S$-strong and if $d(u) + d(v)\geq 2n -1$ for any two
nonadjacent vertices $u, v$ of $S$, then $S$ is cyclable in $D$.
\end{Theorem}

As an extension of cyclability,
we introduce the notion of closed-traceability.
A set $S$ of vertices of a digraph $D$ is said to be {\bf closed-trailable} if $D$ contains a closed ditrail
through all the vertices of $S$.
Clearly, if $S = V (D)$ is closed-trailable, then $D$ is supereulerian.

The following are two results on supereulerian digraphs. The first is due to Bang-Jensen et al. \cite{BaMa14} and the second is due to Algefari et al. \cite{AlLa15}.

\begin{Theorem}\label{BaMa14 2n-3} (\cite{BaMa14})
If a strong digraph $D$ with $n$ vertices satisfying
$d(u)+d(v)\geq 2n-3$ for any two nonadjacent
vertices $u$ and $v$, then $D$ is supereulerian.
\end{Theorem}

\begin{Theorem}\label{AlLa15} (\cite{AlLa15})
If a strong digraph $D$ satisfying
$\lambda(D)\geq \alpha'(D)$, then $D$ is supereulerian.
\end{Theorem}

The purpose of this paper is to analyze two sufficient conditions for closed-trailable subsets in digraphs.
In Section 2,
we prove that for a digraph $D$ with $n$ vertices and a subset $S$ of $V(D)$,
 if $D$ is $S$-strong and if
 $d(u) + d(v)\geq 2n -3$ for any two nonadjacent vertices $u,v$ of $S$, then $S$ is closed-trailable in $D$.
 This result generalizes the Theorem \ref{BaMa14 2n-3}.
 In Section 3,
we show that for a digraph $D$ and a subset $S$ of $V(D)$,
 if $D$ is $S$-strictly strong and if
 $\delta^0(D\langle S\rangle)\geq\alpha'(D\langle S\rangle)>0$, then $S$ is closed-trailable in $D$.
 This result generalizes the Theorem \ref{AlLa15}.

\section{Degree condition}
The following lemma will be needed in this section.
\begin{Lemma}\label{notSx}
Let $D$ be a digraph, $T=u_1u_2\cdots u_h$ be a ditrail in $D$ and $x\in V(D)$.
If $D$ does not contain a $(u_1,u_h)$-ditrail with vertex set $V(T)\cup \{x\}$, then $d_T(x) \leq |V(T)|$.
\end{Lemma}

\begin{proof}
As $D$ does not contain a $(u_1,u_h)$-ditrail with vertex set $V(T)\cup \{x\}$, we have $x\in V(D)-V(T)$ and
$|\{(u_i,x),(x,u_i)\}\cap A(D)|\leq 1$ for any $u_i\in V(T)$. Accordingly, we obtain $d_T(x)\leq |V(T)|$ as required.
\end{proof}

\begin{Theorem}\label{S 2n-3}
Let $D$ be a digraph with $n$ vertices and $S\subseteq V (D)$. If $D$ is $S$-strong and if $d(u) + d(v)\geq 2n -3$ for any two
nonadjacent vertices $u, v$ of $S$, then $S$ is closed-trailable in $D$.
\end{Theorem}
\begin{proof}
Suppose, to the contrary, that
$D$ does not contain a closed ditrail which has all the vertices of $S$.
Since $D$ is $S$-strong,
there exists a closed ditrail which contains at least one vertex of $S$ in $D$.
Let $Q=y_0y_1\cdots y_py_{p+1}\cdots y_{q-1}y_0$ be a closed ditrail with $|V(Q)\cap S|$ maximized in $D$.
Let $|V(Q)|=t$
and $s\in S-V (Q)$.

If $D\langle S\rangle$ contains no nonadjacent vertices, that is, $D\langle S\rangle$ is a semicomplete digraph, then $N_Q(s)\not=\emptyset$.
Assume, w.l.o.g., that $N^+_Q(s)\not=\emptyset$.
Then there exists a vertex $u\in V(Q)$ such that $(s,u)\in A(D)$.
Since $D$ is $S$-strong, for any $s'\in V(Q)\cap S$, there exists a dipath from $s'$ to $s$, and consequently a $(u',s)$-dipath $P$
with
$u'\in V(Q)$ and $V (P )\cap V (Q) = \{u'\}$.
 If $u'=u$, then $Q Psu$ is a closed ditrail with $|V(Q Psu)\cap S|>|V(Q)\cap S|$,
a contradiction.
Thus $u'\not=u$. Then we can obtain a $(u',u)$-dipath $T$ through $s$ such that $V(T)\cap V(Q)=\{u',u\}$.
Choose an $(x,y)$-ditrail $T$ through $s$ with $V(T)\cap V(Q)=\{x,y\}$,
such that the length of the ditrail $P$ is minimum, where $P$ is a $(x,y)$-ditrail
which travels along $Q$ from $x$ to $y$.
If $(V(P)-\{x,y\})\cap S=\emptyset$, then $Q_{[y,x]}T$ is a closed ditrail with $|V(Q_{[y,x]} T)\cap S|>|V(Q)\cap S|$,
a contradiction.
Thus $(V(P)-\{x,y\})\cap S\not=\emptyset$.
Let $s'\in V(P)\cap S$.
Since $D\langle S\rangle$ is a semicomplete digraph,
we have $s$ and $s'$ are adjacent.
If $(s,s')\in A(D)$, then we can get another $(x,s')$-dipath $T'$ through $s$ with $V(T')\cap V(Q)=\{x,s'\}$ such that the length of $(x,s')$-dipath $P'$ in $Q$ is less then the length of $P$ in $Q$, contrary to the choice of $T$ above.
So $(s',s)\in A(D)$. Then we can get another $(s',y)$-dipath $T'$ through $s$ with $V(T')\cap V(Q)=\{s',y\}$ such that the length of $(x,s')$-dipath $P'$ in $Q$ is less then the length of $P$ in $Q$, contrary to the choice of $T$ above.

Therefore $D\langle S\rangle$ contains nonadjacent vertices.
Now we prove that $D$ contains an $(x,y)$-ditrail $T$ through $s$ such that $V(T)\cap V(Q)=\{x,y\}$,
 at most one vertex $t$ in the internal vertices of $T$ with $d^+_T(t)=d^-_T(t)=2$, and other vertex $x$ in the internal vertices of $T$ with $d^+_T(x)=d^-_T(x)=1$.
We will distinguish two cases.
 \vskip .2cm
  \noindent {\bf Case 1.}  $N_Q(s)\not=\emptyset$.

In this case, w.l.o.g., $N^+_Q(s)\not=\emptyset$.
Then there exists a vertex $u\in V(Q)$ such that $(s,u)\in A(D)$.
Since $D$ is $S$-strong, for any $s'\in V(Q)\cap S$, there exists a dipath from $s'$ to $s$, and consequently a $(u',s)$-dipath $P$
with
$u'\in V(Q)$ and $V (P )\cap V (Q) = \{u'\}$.
 If $u'=u$, then $Q Psu$ is a closed ditrail with $|V(Q Psu)\cap S|>|V(Q)\cap S|$,
a contradiction.
Thus $u'\not=u$. Then we can obtain a $(u',u)$-dipath $T$ through $s$ such that $V(T)\cap V(Q)=\{u',u\}$.
 \vskip .2cm
  \noindent {\bf Case 2.}  $N_Q(s)=\emptyset$.

Since $D$ is $S$-strong,
for any $x\in V(Q)\cap S$, there exists a dipath from $s$ to $x$ and a dipath from $x$ to $s$, and consequently a $(s,u_1)$-dipath $P_1$
and a $(u_2,s)$-dipath $P_2$
such that
$u_1,u_2\in V(Q),V (P _1)\cap V (Q) = \{u_1\}$ and $V (P _2)\cap V (Q) = \{u_2\}$.
 Let us choose those paths $P_1$ and $P_2$ such that $|V (P_1)|+|V (P_2)|$
is minimum. To simplify further computations, let $W_1=\{x_1,\cdots,x_{p_1}\}$ be the set of internal vertices of $P_1$,
$W_2=\{y_1,\cdots,y_{p_2}\}$ be the set of internal vertices of $P_2$,
$W_1\cap W_2=L$, $W_1-L=L_1$, $W_2-L=L_2$ and $R=V(D)-V(Q)-L-L_1-L_2-\{s\}$.

If $|L|=0$ or $|L|=1$, then we are done.
So assume that $|L|\geq 2$. Then $|W_1|\geq 2$ and $|W_2|\geq 2$.
Let us choose a vertex $s'\in S\cap V (Q)$ which is by hypothesis nonadjacent to $s$.

Since $|V (P_1)|+|V (P_2)|$
is minimum,
 there is no vertex $r\in R$ satisfying
$\{(s , r),(r, s')\}\subseteq A(D)$ or $\{(s' , r),(r, s)\}\subseteq A(D)$.
Accordingly,
\begin{equation} \label{2|R|}
d_{R}(s)+d_{R}(s')\leq 2|R|=2(n-t-|L\cup L_1\cup L_2|-1).
\end{equation}

Obviously,
\begin{equation} \label{02(t-1)}
d_{Q}(s)=0 \mbox{ and } d_{Q}(s')\leq 2(t-1).
\end{equation}

Since $|V (P_1)|+|V (P_2)|$
is minimum,
we can obtain that $d^+_{L_1}(s)=1$ if $x_1\notin L$ and $0$ otherwise, $d^+_{L}(s)=0$ if $x_1\notin L$ and $1$ otherwise.
Obviously, $d^+_{L_2}(s)\leq |L_2|$.
Thus $d^+_{L\cup L_1\cup L_2}(s)\leq |L_2|+1$.
Similarly, $d^-_{L\cup L_1\cup L_2}(s)\leq |L_1|+1$.
Thus
\begin{equation} \label{s|L_1|+|L_2|+2}
d_{L\cup L_1\cup L_2}(s)\leq |L_1|+|L_2|+2.
\end{equation}
By similar arguments,
we can get that
\begin{equation} \label{s'|L_1|+|L_2|+2}
d_{L\cup L_1\cup L_2}(s')\leq |L_1|+|L_2|+2.
\end{equation}
Combining (\ref{2|R|})-(\ref{s'|L_1|+|L_2|+2}) we get
\[d(s)+d(s')\leq 2n-2|L|\leq 2n-4,\]
but, $s$ and $s'$ are not adjacent and $s,s'\in S$, contradiction.
This proves Case 2.

Choose an $(x,y)$-ditrail $T$ through $s$ with $V(T)\cap V(Q)=\{x,y\}$,
 at most one vertex $t$ in the internal vertices of $T$ with $d^+_T(t)=d^-_T(t)=2$, and other vertex $x$ in the internal vertices of $T$ with $d^+_T(x)=d^-_T(x)=1$,
such that the length of the ditrail $P$ is minimum, where $P$ is a $(x,y)$-ditrail
which travels along $Q$ from $x$ to $y$. 
W.l.o.g., write $T=y_0\cdots s\cdots y_{p+1}$, that is $x=y_0,y=y_{p+1}$.
Let $W = \{y_1,y_2,\cdots, y_p\}$ be the set of internal vertices of $P$,
$T_1=Q_{[y_{p+1},y_0]}$.
Then $|T_1|=t-p+c$, where $c= |W\cap T_1|$.

By the choice of $T$, we have $d_W(s)=0$.
Since if $d_W^+(s)>0$ or $d_W^-(s)>0$, w.l.o.g. $d_W^+(s)>0$, then there exists a vertex $y_j\in W$
($j\in [p]$) such that $(s,y_j)\in A(D)$.
But now we can get another $(y_0,y_j)$-dipath $T'$ 
such that the length of $(y_0,y_j)$-dipath $P'$ in $Q$ is less then the length of $P$ in $Q$, contrary to the choice of $T$ above.
Therefore $d_W^+(s)=0$.
The proof for $d_W^-(s)=0$ is similar.
In particular, $s$ and $y_j$ are nonadjacent, for $j\in [p]$.

If $D$ contains a $(y_{p+1},y_0)$-ditrail $Q'$ with vertex set $V(T_1)\cup \{s\}$,
then $Q' Q_{[y_0,y_{p+1}]}$ is a closed ditrail
with $|V(Q' Q_{[y_0,y_{p+1}]})\cap S|>|V(Q)\cap S|$, contrary to the maximality of $Q$.
Thus $D$ does not have a $(y_{p+1},y_0)$-ditrail with vertex set $V(T_1)\cup \{s\}$.
Then from Lemma \ref{notSx}, we get
\begin{equation} \label{Qs}
d_{Q}(s)\leq d_{W}(s)+d_{T_1}(s)\leq 0+|V(T_1)|=t-p+c.
\end{equation}

If $W$ contains no vertex of $S$, then $T_1T$ is closed ditrail with
$|V(T_1 T)\cap S|>|V(Q)\cap S|$, contrary to the choice of $Q$.
Thus $W$ contains at least one vertex of $S$. Let $W'=W\cap S=\{s_1,s_2,\cdots ,s_l\}$, where $s_1,s_2,\cdots ,s_l$ appear in $P$ in order.
If $D$ contains a $(y_{p+1},y_0)$-ditrail $T_2$ with vertex set
$V(T_2)=V(T_1)\cup W'$, then $T_2T$ is a closed ditrail
with $|V(T_2T)\cap S|>|V(Q)\cap S|$, contrary to the maximality of $Q$.
Thus $D$ does not contain a $(y_{p+1},y_0)$-ditrail $T_2$ with vertex set
$V(T_2)=V(T_1)\cup W'$.
Let $b$ be the maximum integer $1\le i\le l $ such that $D$ contains a $(y_{p+1},y_0)$-ditrail $T_2$ with vertex set
$V(T_2)=V(P_1)\cup \{s_0,s_1,s_2,\cdots ,s_{b-1}\}$, where $s_0=y_0$.
Let $|V(T_2)\cap W|=c'$.
Then $c'\ge c$.
Therefore $D$ does not contain a $(y_{p+1},y_0)$-ditrail with vertex set $V(T_2)\cup \{s_b\}$.
Then from
Lemma \ref{notSx}, we have
\begin{equation} \label{Qs'}
d_{Q}(s_b)= d_{W-T_2}(s_b)+d_{T_2}(s_b)\leq 2(p-1-c')+(t-p+c')= t+p-c'-2
\le t+p-c-2.
\end{equation}

By the choice of $T$ and the maximality of $Q$,
there is no vertex $r\in V(D)-V(Q)$ satisfying $\{(s_b,r), (r,s)\}\subseteq A(D)$
or $\{(s,r),(r,s_b)\}\subseteq A(D)$. Accordingly,
\begin{equation} \label{D-Sss'}
d_{D-Q}(s)+d_{D-Q}(s_b)\leq 2(n-t-1).
\end{equation}

Combining (\ref{Qs'})-(\ref{D-Sss'}),
note that $s$ and $s_b$ are nonadjacent vertices of $S$,
 we can obtain that
\[ d(s) + d(s_b)\leq 2n-4,\]
  contrary to the assumption of Theorem \ref{S 2n-3}.
\end{proof}


\section{Minimum semi-degree condition}

We start with a few lemmas that are needed in our proofs.

\begin{Lemma}\label{not M}
Let $D$ be a digraph, $S\subseteq V (D)$, $H=D\langle S\rangle$ and $M$ be a matching of $H$ with $|M| = m>0$. Suppose that $|S-V(M)| =|X|\geq 2$ and for any
$x \in X$, $d_H(x)\geq 2m- 1$. If there exists a vertex $x'\in X$
such that $d_H(x')\geq 2m + 1$, then $M$ is not a maximum matching of $H$.
\end{Lemma}
\begin{proof}
Assume, by contradiction, that $M$ is a maximum matching of $H$.
By Theorem \ref{M-P iff}, $H$ has no $M$-augmenting path. Let $x,x' \in X$ be two vertices such that
$d_H(x')\geq 2m + 1$ and $d_H(x)\geq 2m - 1$.

Since $H$ has no $M$-augmenting path, $X$ is an independent set in $H$.
So $x$ and $x'$ are nonadjacent in $H$ and $N_H(x),N_H(x')\subseteq V(M)$.
By $d_H(x')\geq 2m + 1$,
there exists an arc $e' = [u', v']\in M$ such that
$[x', u'], [x', v']\in A(H).$ This together with the fact that $H$ has no $M$-augmenting path,
we obtain that
$[x, u'], [x, v']\not\in A(H)$ and $N_H(x)\subseteq V(M)- \{u', v'\}$. Since
$d_H(x)\geq 2m- 1= 2(m -1) + 1$, there exists an arc $[u''
, v'']\in M\setminus e'$ such that $[x, u''], [x, v'']\in A(H)$.
Define
\[M_{x'} = \{[u', v']\in M: [x', u'], [x', v']\in A(H)\} \mbox{ and } M_x = \{[u'', v'']\in M : [x, u''], [x, v'']\in A(H)\}.\]
Then
\begin{equation} \label{4mx}
d_{M_x'}(x')\leq 4|M_{x'}|\mbox{ and }d_{M_x}(x)\leq 4|M_{x}|.
\end{equation}
Let $M' = M - (M_x'\cup M_x)$, possibly $M' =\emptyset$. Then
\begin{equation} \label{2m'}
d_{M'}(x')\leq 2|M'|\mbox{ and }d_{M'}(x)\leq 2|M'|.
\end{equation}
Since $H$ has no $M$-augmenting path, we have $M_x'\cap M_x =\emptyset$, $N_H(x)=N_M(x)\subseteq V(M'\cup  M_x)$ and $N_H(x')=N_M(x')\subseteq V(M'\cup  M_{x'})$.
Combining this, (\ref{4mx}) and (\ref{2m'}),
 we obtain
$ 2m +1\leq d_H(x')=d_{M_x'}(x')+d_{M'}(x') \leq 4|M_{x'}| + 2|M'|$ and $2m - 1 \leq d_H(x)=d_{M_x}(x)+d_{M'}(x) \leq 4|M_x| + 2|M'|$.
Since $4|M_{x'}| + 2|M'|$ and $4|M_x| + 2|M'|$ are even, we have $4|M_{x'}| + 2|M'|\geq 2m + 2$ and $4|M_x| + 2|M'|\geq 2m$. Then we get that
\[4m = 4|M| = (4|M_{x'}| + 2|M'|) + (4|M_v| + 2|M'|) \geq (2m + 2) + 2m = 4m + 2,\]
a contradiction.
\end{proof}

\begin{Lemma}\label{MX jiegou}
Let $D$ be a digraph, $S\subseteq V (D)$, $H=D\langle S\rangle$ and $M$ be a maximum matching of $S$ with $|M| = m>0$. Suppose that $|S-V(M)| =|X|\geq 2$ and
$\delta^0(H)\geq m$.
  Then
  \begin{equation} \label{x=m}
\begin{split}
d^+_H(x)=d^-_H(x)= m\mbox{ for~any } x \in X
\end{split}
  \end{equation}
  Moreover

if $|X|=|\{x,x'\}|=2$ and $M_x=\frac{m}{2}$ or $M_{x'}=\frac{m}{2}$, where $M_x=\{[u,v]\in M:[u,x]\mbox{ or }[v,x]\in A(H)\},M_{x'}=\{[u,v]\in M:[u,x']~ or~ [v,x']\in A(H)\}$, then 
$D\langle V(M_x) \cup \{x\}\rangle\cong D\langle V(M_{x'}) \cup \{x'\}\rangle\cong K^*_{m+1}$ and
$N_{V(M_x) \cup \{x\}}(V(M_{x'}) \cup \{x'\}) = \emptyset$.
 Otherwise, for any $e = [u, v]\in M$
and for any $x \in X$,
each of the following holds.

(i) There exists exactly one $v(e) \in \{u, v\}$ such that $\{(v(e), x), (x, v(e))\}\subseteq A(D)$ and
 $N_X(u(e))=\emptyset$, where $u(e) \in \{u, v\} - \{v(e)\}$.

(ii) The set $\{u(e) : e \in M\}$ is an independent set in $H$ with
$d_H^+(u(e)) =d_H^-(u(e)) = m$ and $\{(u(e), v(e')),
(v(e'), u(e))\}\subseteq A(D)$ for any $e, e' \in M$ (possibly $e= e'$).
\end{Lemma}
\begin{proof}
By Theorem \ref{M-P iff} and $M$ is a maximum matching, we have that,
  \begin{equation} \label{M-P}
H\mbox{ does not contain an } M\mbox{-augmenting path. }
 \end{equation}
By (\ref{M-P}), $X$ is an independent set in $H$ and $N_H(X)\subseteq V(M)$. That is $N_H(X)=N_M(X)$.

 Since $M$ is a maximum matching,
it follows by $\delta^0(H)\geq m$ and by Lemma \ref{not M} that (\ref{x=m}) must hold.
For the remaining part of this lemma, we consider two cases. \vskip .2cm
  \noindent {\bf Case 1.} $|X|=|\{x,x'\}|=2$.

Let $M = \{e_1, \cdots , e_m\}$, $M_x=\{[u,v]\in M:[u,x]\mbox{ or }[v,x]\in A(H)\}$ and $M_{x'}=\{[u,v]\in M:[u,x']\mbox{ or } [v,x']\in A(H)\}$. By (\ref{x=m}), $|M_x| \geq \frac{m}{2}$
and $|M_{x'}| \geq\frac{m}{2}$.
If $|M_x| = \frac{m}{2}$
or $|M_{x'}|=\frac{m}{2}$, then $m$ must be even. W.l.o.g., assume
that $M_x = \{e_1, \cdots , e_{\frac{m}{2}}\}$.
By (\ref{x=m}), for any $[u,v]\in M_x$, $\{(u,x),(x,u),(v,x),(x,v)\}\subseteq A(H)$.
By (\ref{M-P}), we have $M_{x'} = \{e_{\frac{m}{2}+1}, \cdots, e_m\}$ and for any $u\in V(M_x)$ and $u' \in V(M_{x'})$, $[u, u']\not\in A(H)$.
Therefore
$N_{V(M_x) \cup \{x\}}(V(M_{x'}) \cup \{x'\}) = \emptyset$.
By $\delta^0(H)\geq m$, we can conclude that $\{(u, v), (v, u)\}\subseteq A(H)$ for any $u,v\in V(M_x\cup \{x'\})$ and
$\{(u', v'), (v', u')\}\subseteq A(H)$ for any $u',v'\in V(M_{x'}\cup \{x'\})$.
Therefore $D\langle V(M_x) \cup \{x\}\rangle\cong D\langle V(M_{x'}) \cup \{x'\}\rangle\cong K^*_{m+1}$.
This proves the case of $|X|=|\{x,x'\}|=2$ and $M_x=\frac{m}{2}$ or $M_{x'}=\frac{m}{2}$.

Now we consider the case of $|M_x| > \frac{m}{2}$
and $|M_{x'}|>\frac{m}{2}$ for $|X|=|\{x,x'\}|=2$.

Then $M_x \cap M_{x'} \not= \emptyset$. Let $M' = M_x \cap M_{x'} = \{e'_i= [u'_i, v'_i]:i \in [d], d\in [m])\}\subseteq M$.
By (\ref{M-P}), we get that for any arc $e'_i= [u'_i, v'_i] \in M'$, exactly one vertex of $e'_i$ can be adjacent to both $x$ and $x'$,
w.l.o.g., assume that
\begin{equation} \label{(A)}
\begin{split}
[u'_i, x], [u'_i, x'] \in A(H)\mbox{ and }[v'_i, x], [v'_i, x'] \not\in A(H)
\mbox{ for~any }e'_i= [u'_i, v'_i] \in M'.
\end{split}
\end{equation}
Then
\begin{equation} \label{(D)}
\begin{split}
\mbox{for any }e'_i = [u'_i, v'_i]\in M', d_{\{u'_i, v'_i\}}(x)\leq 2\mbox{ and }d_{\{u'_i, v'_i\}}(x')\leq 2.
\end{split}
\end{equation}

If $[v'_i, v'_j]\in A(H)$ for $i,j \in [d]$, then $\{[u'_i,x],[u'_i,v'_i],[v'_i, v'_j],[v'_j, u'_j],[u'_j,x']\}$ induce an $M$-augmenting path,
which contradicts (\ref{M-P}).
Thus
\begin{equation} \label{(B)}
\begin{split}
\mbox{the set }\{v'_1, v'_2, \cdots , v'_d\}\mbox{ is an independent set in }H.
\end{split}
\end{equation}

Let $M'' = M - M' =(M_x-M')\cup(M_x'-M')= \{e''_j= [u''_j, v''_j]: j\in[ m - d]\}$.
Then we have the following claim.
\\
  \noindent {\bf Claim 1.1} $d=m$, that is $M_x-M'=\emptyset$ and $ M_x'-M'=\emptyset$.

By contradiction, assume that $d<m$, that is $M''\not=\emptyset$.
W.l.o.g., assume that $M_x-M'\not=\emptyset$.
By the definition of $M_x$ and $ M_{x'}$,
for any $e=[u,v]\in M_x-M'$, $d_{\{u,v\}}(x)\leq 4$, and for any $e'=[u',v']\in M_x'-M'$, $d_{\{u',v'\}}(x')\leq 4$.
If there exists
some $e=[u,v]\in M_x-M'$ such that $[x, u] \in A(H)$ and $[x, v] \not\in A(H)$,
then $d_{\{u,v\}}(x)\leq 2$.
Let $|M_x'-M'|=d'$.
By (\ref{x=m}) and
(\ref{(D)}), we have
\[
2m =d_H( x)= d_M( x)
= d_{M'}(x)
+d_{ \{u, v\}}(x)+ d_{M_x-M'- \{u, v\}}( x)\leq 2d+2+4(m-d-d'-1),
\]
\[
2m =d_H( x')= d_M( x')
= d_{M'}(x')
+ d_{M_{x'}-M'}( x')\leq 2d+4d',
\]
Combining the last two inequalities, we obtain $4m\leq 4m-2$,
a contradiction. Thus
\begin{equation} \label{(E)}
\begin{split}
\mbox{for any }e=[u,v]\in M_x-M', [x, u] \in A(H)
\mbox{ and }[x, v] \in A(H).
\end{split}
\end{equation}
By similar arguments, we can get that
for any $e'=[u',v']\in M_{x'}-M'$, $[x', u'] \in A(H)$
 and $[x', v'] \in A(H).$

If for some $[u'_i, v'_i] \in M'$, $N_H(v'_i)\cap N_{M_x-M'}(v'_i)\not=\emptyset$, w.l.o.g, assume that
$u\in V(M_x-M')$ with $e=[u,v]\in M_x-M'$ such that $[u,v'_i]\in A(H)$, then by (\ref{(E)}),
$\{[x, v],[u, v],[ u,v'_i],[u'_i, v'_i],$\\ $[u'_i,x']\}$ induce an $M$-augmenting path,
which contradicts (\ref{M-P}).
Thus \begin{equation} \label{(F)}
\begin{split}
\mbox{for any }[u'_i, v'_i] \in M', N_H(v'_i)\cap N_{M_x-M'}(v'_i)=\emptyset.
\end{split}
\end{equation}
By similar arguments, we can get that
\begin{equation} \label{(F')}
\begin{split}
\mbox{for any }[u'_i, v'_i] \in M', N_H(v'_i)\cap N_{M_x'-M'}(v'_i)=\emptyset.
\end{split}
\end{equation}

Since $M'\not=\emptyset, d\geq 1$,
 $v'_1$ exists. By (\ref{(A)}), (\ref{(B)}), (\ref{(F)}) and (\ref{(F')}), $N_H(v'_1)\subseteq \{u'_1, u'_2,\cdots , u'_d\}$. Then $d_H(v'_1) \leq 2d$.
  By $\delta^0(H)\geq m$ and $d\leq m-1$,
we have $2m \leq d_H(v'_1)\leq 2d \leq 2(m - 1)$, a contradiction. Thus $d = m$.
This proves Claim 1.1.

(i) By Claim 1.1, we have
$M = M'$.
By the fact that $[u'_i, x], [u'_i, x'] \in A(H),[v'_i, x], [v'_i, x'] \not\in A(H)$ for any $e_i=[u'_i,v'_i]\in M'$
with (\ref{x=m}), we have $(u'_i,x), (u'_i, x'), (x, u'_i ), (x', u'_i) \in A(H)$. This proves (i) for the case of $|X|=\{x,x'\}=2,|M_x| > \frac{m}{2}$
and $|M_{x'}|>\frac{m}{2}$.

(ii) By Claim 1.1 with (\ref{(A)}) and (\ref{(B)}), $N_H(v'_1)\subseteq \{u'_1, u'_2,\cdots , u'_m\}$.
Then by $\delta^0(H)\geq m$, we have that $d^+(v'_i) =
d^-(v'_i) = m$ for any $i\in [m]$ and $(u'_i, v'_{i'} ), (v'_{i'} , u'_i) \in A(H)$ for any $ i, i' \in [m]$.
This proves (ii) for the case of $|X|=\{x,x'\}=2,|M_x| > \frac{m}{2}$
and $|M_{x'}|>\frac{m}{2}$.
 \vskip .2cm
  \noindent {\bf Case 2.} $|X|\geq 3$.

  If $m=1$, then (i) and (ii) are easy to verity.
  Assume that $m=2$.
   If there exists $e_1=[u_1,v_1]\in M$ such that $[u_1,x_1],[x_1,v_1]\in A(H)$ for some $x_1\in X$,
  then by (\ref{M-P}), we have $[u_1,x],[x,v_1]\not\in A(H)$ for any $x\in X-x_1$.
  Since $d^+_H(x)=d^-_H(x)=(m-1)+1$ for any $x\in X-x_1$, we have $[u_2,x],[x,v_2]\in A(H)$ for $e_2=[u_2,v_2]= M-e_1$.
  As $|X|\ge 3$, there are at least two distinct vertices $x_2,x_3\in X-x_1$ such that
  $[u_2,x_2],[x_2,v_2],[u_2,x_3],[x_3,v_2]\in A(H)$.
  Then $\{[u_2,x_2],[u_2,v_2],[x_3,v_2]\}$ induce an $M$-augmenting path, which contradicts (\ref{M-P}).
  If there exists $e=[u,v]\in M$ such that $[u,x_1],[x_1,v]\notin A(H)$ for some $x_1\in X$, then by $d^+_H(x_1)=d^-_H(x_1)=(m-1)+1$,
  there exists $e_1=[u_1,v_1]\in M-e$ such that $[u_1,x_1],[x_1,v_1]\in A(H)$.
  Then by above, we can get a contradiction.
  Thus for any $e=[u,v]\in M$ and any $x\in X$, there exists exactly one $v(e) \in \{u, v\}$ such that $[v(e), x]\in A(D)$ and
 $N_X(u(e))=\emptyset$, where $u(e) \in \{u, v\} - \{v(e)\}$.
 Since $d^+_H(x)=d^-_H(x)= m$ for any $x \in X$,
 $\{(v(e), x), (x, v(e))\}\subseteq A(D)$. This proves (i) for the case of $m=2$.
 By (\ref{M-P}),
the set $\{u(e) : e \in M\}$ is an independent set in $H$.
Then by $\delta^0(H)\geq m$,
$d_H^+(u(e)) =d_H^-(u(e)) = m$ and $\{(u(e), v(e')),
(v(e'), u(e))\}\subseteq A(D)$ for any $e, e' \in M$ (possibly $e= e'$).
This proves (ii) for the case of $m=2$.

 Now we assume that $m\ge 3$.
 If there exists $e_1=[u_1,v_1]\in M$ such that $[u_1,x_1],[x_1,v_1]\in A(H)$ for some $x_1\in X$,
  then by (\ref{M-P}), we have $[u_1,x],[x,v_1]\not\in A(H)$ for any $x\in X-x_1$.
  Since $d^+_H(x)=d^-_H(x)=(m-1)+1$ for any $x\in X-x_1$, there exists $e_2=[u_2,v_2]\in M-e_1$ such that $[u_2,x_2],[x_2,v_2]\in A(H)$
  for some $x_2\in X-x_1$.
  By (\ref{M-P}), for any $x'\in X-x_1-x_2$, $[u_2,x'],[x',v_2]\not\in A(H)$.
  Since $d^+_H(x')=d^-_H(x')=(m-2)+2$ for any $x'\in X-x_1-x_2$, there exists $e_3=[u_3,v_3]\in M-e_1-e_2$ such that $[u_3,x_3],[x_3,v_3]\in A(H)$
  for some $x_3\in X-x_1-x_2$.
  Let $M_i=\{[u,v]\in M:[u,x_i],[v,x_i]\in A(H)\}$ for $i\in[3]$.
  By (\ref{M-P}), any two of $\{M_1,M_2,M_3\}$ are disjoint.
  Let $M'=M-M_1-M_2-M_3$.
  Then we have
  \[m=d^+_H(x_1)=d^+_{M_1}(x_1)+d^+_{M'}(x_1)\le 2|M_1|+m-|M_1|-|M_2|-|M_3|\] and
\[m=d^+_H(x_2)=d^+_{M_2}(x_2)+d^+_{M'}(x_2)\le 2|M_2|+m-|M_1|-|M_2|-|M_3|.\]
Hence we get $-2|M_3|\geq 0$, a contradiction as $|M_3|\ge 1$.
If there exists $e=[u,v]\in M$ such that $[u,x_1],[x_1,v]\notin A(H)$ for some $x_1\in X$, then by $d^+_H(x_1)=d^-_H(x_1)=(m-1)+1$,
  there exists $e_1=[u_1,v_1]\in M-e$ such that $[u_1,x_1],[x_1,v_1]\in A(H)$.
  Then by above, we can get a contradiction.
  Thus for any $e=[u,v]\in M$ and any $x\in X$, there exists exactly one $v(e) \in \{u, v\}$ such that $[v(e), x]\in A(D)$ and
 $N_X(u(e))=\emptyset$, where $u(e) \in \{u, v\} - \{v(e)\}$.
 Since $d^+_H(x)=d^-_H(x)= m$ for any $x \in X$,
 $\{(v(e), x), (x, v(e))\}\subseteq A(D)$. This proves (i) for the case of $m\ge3$.
 By (\ref{M-P}),
the set $\{u(e) : e \in M\}$ is an independent set in $H$.
Then by $\delta^0(H)\geq m$,
$d_H^+(u(e)) =d_H^-(u(e)) = m$ and $\{(u(e), v(e')),
(v(e'), u(e))\}\subseteq A(D)$ for any $e, e' \in M$ (possibly $e= e'$).
This proves (ii) for the case of $m\ge3$.
This completes the proof for Case 2, and the proof of Lemma \ref{MX jiegou} as well.
\end{proof}

\begin{Theorem}\label{min-semi matching}
Let $D$ be a digraph, $S\subseteq V (D)$, $H=D\langle S\rangle$, $M$ be a maximum matching of $H$ with $|M| = m>0$ and $S-V(M) =X$.
Suppose that 
if $|X|=|\{x,x'\}|=2$ and $M_x=\frac{m}{2}$ or $M_{x'}=\frac{m}{2}$, where $M_x=\{[u,v]\in M:[u,x]\mbox{ or }[v,x]\in A(H)\}$ and $M_{x'}=\{[u,v]\in M:[u,x']\mbox{ or } [v,x']\in A(H)\}$, then $D$ is $S$-strictly strong.
 If $\delta^0(H)\geq m$,
then $S$ is closed-trailable.
\end{Theorem}
\begin{proof}
If $|X| = 0$ or $|X| =  1$, that is $|S| = 2m$ or $|S| = 2m + 1$, then by $\delta^0(H)\geq m$, we have that $H$ is strong.
In fact, if $H$ is not strong, then there exist two vertices $u,v\in V(H)$ such that there is no dipath from $u$
to $v$ in $H$.
Then $(u,v)\not\in A(H)$ and $d^+_H(u)+ d^-_H(v)\leq |S|-2$.
By $\delta^0(H)\geq m$, we get $d^+_H(u)\geq m$ and $d^-_H(v)\geq m$.
Then $2m\le d^+_H(u)+ d^-_H(v)\le |S|-2$,
a contradiction.
Thus $H$ is strong.
If $H$ contains no nonadjacent vertices, that is, $H$ is a semicomplete digraph, then by
by Theorem 1.5.3 of \cite{BaGu09},
$H$ is hamiltonian and so supereulerian.
Thus $S$ is closed-trailable in $D$.
Therefore $H$ contains nonadjacent vertices.
Then by
Theorem \ref{BaMa14 2n-3} with the fact that $\delta^0(H)\geq m$, $H$ is supereulerian, and so $S$ is closed-trailable in $D$ and we are done.
Thus we assume that $|S|\geq2m + 2$ in the following.

If $|X|=\{x,x'\}=2$ and $M_x=\frac{m}{2}$ or $M_{x'}=\frac{m}{2}$, where $M_x=\{[u,v]\in M:[u,x]\mbox{ or }[v,x]\in A(H)\}$ and $M_{x'}=\{[u,v]\in M:[u,x']\mbox{ or } [v,x']\in A(H)\}$, then by
Lemma \ref{MX jiegou}, we have $D\langle V(M_x) \cup \{x\}\rangle\cong D\langle V(M_{x'}) \cup \{x'\}\rangle\cong K^*_{m+1}$ and
$N_{V(M_x) \cup \{x\}}(V(M_{x'}) \cup \{x'\}) = \emptyset$.
Then $H$ is not connected and $H$ is not a semicomplete digraph.
If $S=V(D)$, then by $D$ is $S$-strictly strong, we have $D=H$ is connected,
a contradiction.
Thus $S\subsetneq V(D)$.
Since $D$ is $S$-strictly strong, then there exist
 two nonadjacent
vertices $u,v\in S$ such that $D$ contains a closed ditrail $Q$ through the vertices $u$ and $v$.
Assume, w.l.o.g., that $u\in V(M_x) \cup \{x\}$ and $v\in V(M_{x'}) \cup \{x'\}$.
Let $Q_{[x_h,x_{h'}]}$ be the sub-ditrail of $Q_{[u,v]}$ such that
$V(Q_{[x_h,x_{h'}]})\cap (V(M_x) \cup \{x\})=\{x_h\}$ and $V(Q_{[x_h,x_{h'}]})\cap (V(M_{x'}) \cup \{x'\})=\{x_{h'}\}$, and $Q_{[x_{l'},x_{l}]}$
be the sub-ditrail of $Q_{[v,u]}$ such that $V(Q_{[x_{l'},x_{l}]})\cap (V(M_x) \cup \{x\})=\{x_l\}$ and $V(Q_{[x_{l'},x_{l}]})\cap (V(M_{x'}) \cup \{x'\})=\{x_{l'}\}$.
Since $D\langle V(M_x) \cup \{x\}\rangle\cong D\langle V(M_{x'}) \cup \{x'\}\rangle\cong K^*_{m+1}$, there exist a spanning ditrail $T_{[x_l,x_{h}]}$ from $x_l$ to $x_h$ in $D\langle V(M_x) \cup \{x\}\rangle$ and a spanning ditrail $T'_{[x_{h'},x_{l'}]}$ from $x_{h'}$ to $x_{l'}$ in $D\langle V(M_{x'}) \cup \{x'\}\rangle$.
Then $Q_{[x_h,x_{h'}]}  T'_{[x_{h'},x_{l'}]}  Q_{[x_{l'},x_{l}]}   T_{[x_l,x_{h}]}$
is a closed ditrail in $D$ which
contains all the vertices of $S$.

Now we consider the remaining
case of this theorem.
Since $\delta^0(H)\geq m\geq 1$,
there exists a closed ditrail which contains the vertices of $S$ in $D$. Let $Q$ be a closed ditrail of $D$ with
\begin{equation} \label{many S k}
\begin{split}
|V(Q)\cap S|\mbox{ is maximized}.
\end{split}
\end{equation}

If $ S \subseteq V(Q)$, then $S$ is closed-trailable in $D$ and we are
done. So assume that $S-V(Q)\not=\emptyset$. Then we consider two cases.
 \vskip .2cm
  \noindent {\bf Case 1.} $X-V(Q)\not=\emptyset$.

Let $x\in X-V(Q)$.
If there exists $e =[u, v] \in M$ such that $u, v\in V(Q)$,
then by Lemma
\ref{MX jiegou}(i), there exists a vertex
$v(e) \in \{u, v\}$ such that $(x, v(e)), (v(e), x) \in A(H)$. Then $Q \cup
\{(x, v(e)), (v(e), x)\}$ is a closed ditrail $Q'$ of $D$ with
$|V(Q')\cap S| >|V(Q)\cap S|$, which contradicts (\ref{many S k}).
Thus for any $e =[u, v] \in M$, $|\{u, v\}\cap V(Q)|\leq 1$.
If $|\{u, v\}\cap V(Q)|= 1$ for some $[u, v] \in M$, w.l.o.g., $u\notin V(Q)$ and $v\in V(Q)$,
then by Lemma
\ref{MX jiegou}(ii), $(u, v), (v, u) \in A(H)$. But now $Q \cup
\{(u, v), (v, u)\}$ is a closed ditrail $Q'$ of $D$ with
$|V(Q')\cap S| >|V(Q)\cap S|$, which contradicts (\ref{many S k}).
Therefore $|\{u, v\}\cap V(Q)|= 0$ for any $e =[u, v] \in M$.
Thus $V(M)\cap V(Q)=\emptyset$ and $X\cap V(Q)\not=\emptyset$.
Pick $s\in X\cap V(Q)$, then by Lemma
\ref{MX jiegou}(i), there exists a vertex
$v(e) \in \{u, v\}$ such that $(s, v(e)), (v(e), s) \in A(H)$. Then $Q \cup
\{(s, v(e)), (v(e), s)\}$ is a closed ditrail $Q'$ of $D$ with
$|V(Q')\cap S| >|V(Q)\cap S|$, which contradicts (\ref{many S k}).
 \vskip .2cm
  \noindent {\bf Case 2.} $X\subseteq V(Q)$.

Let $v \in S-V(Q)$. Then $v \in V(M)$ and there exists an arc $e =
[u, v] \in M$. If $u \in V(Q)$, then by Lemma \ref{MX jiegou}(ii), $\{(u, v), (v, u)\} \subseteq A(H)$,
then we have a closed ditrail $Q'=Q \cup \{(u, v), (v, u)\}$ of $D$ with
$|V(Q')\cap S| >|V(Q)\cap S|$, which contradicts (\ref{many S k}). Thus $u \not \in V(Q)$.
Since $|X|\geq 2$, there exists a vertex $w \in X=S-V(M)$. Since $S-V(Q)\subseteq V(M)$, we have $ w \in V(Q)$. Then by Lemma \ref{MX jiegou}(i), there
exists a vertex $v(e) \in \{u, v\}$ such that $(w, v(e)), (v(e), w) \in A(D)$. Now $Q \cup \{(w, v(e)), (v(e), w)\}$ is a closed ditrail $Q'$ of $D$ with
$|V(Q')\cap S| >|V(Q)\cap S|$, which contradicts (\ref{many S k}). This proves Theorem
\ref{min-semi matching}.
 \end{proof}

 \begin{Corollary}
 Let $D$ be a digraph and $S$ be a subset of $V(D)$.
 If $D$ is $S$-strictly strong and
 $\delta^0(D\langle S\rangle)\geq\alpha'(D\langle S\rangle)>0$, then $S$ is closed-trailable in $D$.
\end{Corollary}
%
%
%






\end{document}